\documentclass[11pt,leqno,twoside]{amsart}
\usepackage{amssymb,amsmath,amsthm,soul,color}
\usepackage{t1enc}
\usepackage[cp1250]{inputenc}
\usepackage{a4,indentfirst,latexsym}
\usepackage{graphics}
\usepackage{mathrsfs}
\usepackage{cite,enumitem,graphicx}
\usepackage[colorlinks=true,urlcolor=blue,
citecolor=red,linkcolor=blue,linktocpage,pdfpagelabels,
bookmarksnumbered,bookmarksopen]{hyperref}
\usepackage[english]{babel}
\usepackage[left=2.61cm,right=2.61cm,top=2.72cm,bottom=2.72cm]{geometry}
\usepackage[metapost]{mfpic}
\usepackage[hyperpageref]{backref}
\usepackage[colorinlistoftodos]{todonotes}
\usepackage[normalem]{ulem}

\allowdisplaybreaks

\makeatletter
\providecommand\@dotsep{5}
\def\listtodoname{List of Todos}
\def\listoftodos{\@starttoc{tdo}\listtodoname}
\makeatother

\newtheorem{Th}{Theorem}[section]
\newtheorem{Prop}[Th]{Proposition}
\newtheorem{Lem}[Th]{Lemma}

\newtheorem{Def}[Th]{Definition}

\newcommand{\vp}{\varphi}

\newcommand{\eps}{\varepsilon}

\def\div{\mathop{\mathrm{div}\,}}

\def\supp{\mathrm{supp}}

\def\id{\mathrm{id}}

\def\N{\mathbb{N}}
\def\R{\mathbb{R}}

\def\curlop{\nabla\times}
\def\J{\mathcal{J}}

\def\D{\mathcal{D}}

\def\A{\mathcal{A}}

\newcommand{\cA}{{\mathcal A}}

\newcommand{\cC}{{\mathcal C}}
\newcommand{\cD}{{\mathcal D}}

\newcommand{\cF}{{\mathcal F}}

\newcommand{\cJ}{{\mathcal J}}

\newcommand{\cL}{{\mathcal L}}
\newcommand{\cM}{{\mathcal M}}

\newcommand{\cO}{{\mathcal O}}

\newcommand{\cS}{{\mathcal S}}

\newcommand{\al}{\alpha}
\newcommand{\be}{\beta}
\newcommand{\ga}{\gamma}

\newcommand{\Om}{\Omega}

\def\curlop{\nabla\times}
\newcommand{\weakto}{\rightharpoonup}

\def\id{\mathrm{id}}

\newenvironment{altproof}[1]
{\noindent%\addvspace{0.3cm}
	{\em Proof of {#1}}.}
{\nopagebreak\mbox{}\hfill $\Box$\par\addvspace{0.5cm}}

\numberwithin{equation}{section}

\begin{document}
	\title[Between Maxwell and Born-Infeld]{Between Maxwell and Born-Infeld:\\ the presence of the magnetic field}
	\author[P. d'Avenia]{Pietro d'Avenia}
	\author[J. Mederski]{Jaros\l aw Mederski}
	\address[P. d'Avenia]{\newline\indent
		Dipartimento di Meccanica, Matematica e Management
		\newline\indent 
		Politecnico di Bari
		\newline\indent
		Via Orabona 4,  70125  Bari, Italy}
	\email{\href{mailto:pietro.davenia@poliba.it}{pietro.davenia@poliba.it}}
	
	\address[J. Mederski]{\newline\indent 
		Institute of Mathematics,
		\newline\indent
		Polish Academy of Sciences,
		\newline\indent 
		ul. \'Sniadeckich 8, 00-956
		Warszawa, Poland}
	\email{\href{mailto:jmederski@impan.pl}{jmederski@impan.pl}}

	\date{\today}
	
	\begin{abstract}
		Our motivation is to consider an electromagnetic Lagrangian density $\mathcal{L}_q$, depending on a parameter such that, for $q=1$ it corresponds to the Born-Infeld Lagrangian density and for $q=2$ it restores the Maxwell one. The model in the presence of given charge and current densities is investigated. We solve the pure magnetostatic problem for $q\in(6/5,2)$. We also study  the electrostatic problem in the presence of an assigned magnetic field for $q\in[1,2)$.
	\end{abstract}
	
	\maketitle

	\section{Introduction}\label{sec:intro}
	
	In the 1930's, almost all physicists was adopting the so called {\em dualistic} point of view in order to interpret the relation between matter and electromagnetic field: the particle are the sources of the field and interact with it, but they are not part of the field.
	This was  essentially due to the failure of any attempt to develop an unitarian theory (which, roughly speaking, states that in the nature there exists only the electromagnetic field and the particles are singularities of the field), to the results of the Relativity Theory (and in particular the dependence of the mass on the velocity, which is not characteristic of electromagnetic mass and can be derived only from the transformation law) and to the great success of the Quantum Mechanics (which starts exactly from the consideration of oscillators and particles moving in a Coulomb field).
	But this approach met some difficulties essentially due to the fact that the energy of a point charge is {\em infinite}.
	
	These considerations and the belief on the unitarian approach
	from a philosophical point of view, led Born \cite{Bnat,B} and then Born and Infeld \cite{BInat,BI} to construct a new theory of the electromagnetic field introducing, respectively, the lagrangian densities
	\[
	\mathcal{L}_{\rm B}
	=b^2 \left(1-\sqrt{1-\frac{| E |^2 - | B |^2}{b^2}}\right)
	=b^2 \left(1-\sqrt{1-\frac{|\nabla\phi+\partial_t  A |^2 - |\nabla\times  A |^2}{b^2}}\right)
	\]
	and
	\begin{align*}
		\mathcal{L}_{\rm BI}
		&=
		b^2 \left(1-\sqrt{1-\frac{| E |^2 - | B |^2}{b^2}-\frac{( E \cdot B )^2}{b^4}}\right)\\
		&=
		b^2 \left(1-\sqrt{1-\frac{|\nabla\phi+\partial_t  A |^2 - |\nabla\times  A |^2}{b^2}-\frac{[(\nabla\phi+\partial_t  A )\cdot\nabla\times  A ]^2}{b^4}}\right),
	\end{align*}
	where $b$ is a constant and $E$ and $B$ are the electric and the magnetic field in the space time $\mathbb{R}\times\mathbb{R}^3$ whose expression, through the gauge potentials $\phi$ and $A$, is
	\[
	E= -\nabla\phi - \partial_t  A ,
	\qquad
	B=\nabla\times A .
	\]
	As explained in \cite{BI}, Born and Infeld started from the following postulate:
	\begin{center}
		{\em The action integral has to be an invariant.}
	\end{center}
	The {\em action} is written usually as	\[
	\mathcal{S}=\int \mathcal{L}
	\]
	where $\mathcal{L}$ is the Lagrangian density.\\
	Now, if  $\mathcal{L}$ is a function of an arbitrary covariant tensor field $a_{kl}$, to get the required invariance, by \cite{Eddington}, $\mathcal{L}$ must be $\sqrt{\operatorname{det}(a_{kl})}$.
	Indeed, following the arguments in \cite[page 429]{BI} (see in particular the first footnote where the authors refer to \cite[\S 48 and \S 101]{Eddington}), if we consider a transformation with Jacobian $J$, since $	\operatorname{det}(\widetilde{a}_{kl})=J^{-2}\operatorname{det}(a_{kl})$, where $\widetilde{a}_{kl}$'s are the transformations of $a_{kl}$'s, then
	\[
	\int \sqrt{\operatorname{det}(\widetilde{a}_{kl})} \, d\widetilde{x}
	=
	\int \frac{1}{|J|}\sqrt{\operatorname{det}(a_{kl})} |J|\, dx
	=
	\int \sqrt{\operatorname{det}(a_{kl})}\, dx.
	\]

	Now, to consider simultaneously the metrical and the electromagnetic field, as Einstein in 1923 and 1925, Born and Infeld started from a unique tensor field $a_{kl}$, identifying its symmetrical part $g_{kl}$ as the metrical field and its antisymmetrical part $f_{kl}$ as the electromagnetic field, obtaining the following three invariant densities
	\[
	\sqrt{-\operatorname{det}(a_{kl})},
	\quad
	\sqrt{-\operatorname{det}(g_{kl})},
	\quad
	\sqrt{\operatorname{det}(f_{kl})}
	\]
	(the minus signs are due to the signature of the metric tensor), and so they took the simplest form including these three function, namely a linear combination
	\[
	\mathcal{L}=
	\sqrt{-\operatorname{det}(a_{kl})}
	+A
	\sqrt{-\operatorname{det}(g_{kl})}
	+B
	\sqrt{\operatorname{det}(f_{kl})}.
	\]
	Then, if $f_{kl}$ is the rotation of a potential vector, the last term can be omitted.
	Moreover, to have the classical Maxwell Lagrangian density in the limiting case for small values of $f_{kl}$, they took $A=-1$.\\
	As discussed above, the particular shape of $\cL_{\rm BI}$ allows Born and Infeld to get the invariance of their action for all transformations.
	
	Note that in \cite[Section 4]{BB}, Bialynicki-Birula criticized  the Born Infeld motivation stating: \guillemotleft{{\em Every function of $S[=(|E|^2-|B|^2)/2]$ and $P[=E\cdot B]$ can be easily converted into a scalar under all coordinates transformations with the use of the metric tensor}.}\guillemotright.

	If we want to approach the Born-Infeld theory from a variational point of view, the behaviour at infinity of the Lagrangian density is an obstacle. Thus we propose a new model, considering a modified version of such a Lagrangian density which is obtained {\em interpolating} by a power $q$ the Born Infeld theory with the classical Maxwell one, namely
	\begin{equation*}%\label{Lq}
		\begin{split}
			\mathcal{L}_q
			&=
			\frac{b^2}{q} \left(1-\Big(1-\frac{| E |^2 - | B |^2}{b^2}-\frac{( E \cdot B )^2}{b^4}\Big)^{q/2}\right)\\
			&=
			\frac{b^2}{q} \left(1-\Big(1-\frac{|\nabla\phi+\partial_t  A |^2 - |\nabla\times  A |^2}{b^2}-\frac{[(\nabla\phi+\partial_t  A )\cdot\nabla\times  A ]^2}{b^4}\Big)^{q/2}\right),
		\end{split}
	\end{equation*}
	where $q\in [1,2]$.
	Clearly $\cL_{1}=\cL_{BI}$, instead $\cL_2$ corresponds to the Maxwell theory with the additional term depending on $( E \cdot B )^2/b^2$ that, to recover the classical Maxwell theory, can be seen as a negligible term.

	The primary motivation for introducing the interpolated model $\mathcal{L}_q$, with $1<q<2$, is to bridge the gap between classical Maxwell theory and Born--Infeld theory, combining the main strengths of both while overcoming their respective limitations. On the one hand, classical Maxwell theory ($q=2$) enjoys good analytical properties due to its quadratic growth, but it suffers from the physically undesirable feature of infinite energy for point charges. On the other hand, Born--Infeld theory ($q=1$) resolves this divergence by yielding finite energy configurations, but its linear growth at infinity leads to serious analytical difficulties, in particular the lack of reflexivity of the associated functional spaces, which prevents the use of standard variational methods.
	
	The interpolated model inherits from Born--Infeld theory the physically relevant property of finite energy for point charges for $q\in[1,2)$, while exhibiting superlinear growth. As we shall see later, the assumption $q>6/5$ plays a crucial role here, as it restores the reflexivity of the functional framework and allows us to apply rigorous variational techniques that are unavailable in the pure Born--Infeld case.
	
	Moreover, taking into account the previous considerations, the model enjoys good invariance properties. Indeed, since we consider the Lagrangian density in the form $(\det(a_{kl}))^{q/2}$, arguing as above, we get that the associated action is invariant under transformations with $|J|=1$, including, in particular, Lorentz and Poincar\'e transformations.

	A natural question may be raised  concerning the existence of a solution representing an electrostatic field  in the presence of a fixed magnetic field $B$ or a magnetic field in the absence of an external electric field $E$. The electrostatic case for the Born-Infeld theory $\cL_1$ with $A=0$ leads to nonlinear equations and has attracted a considerable attention in the recent literature, see \cite{DenisPietroAlessio,BDPR,BIa1,BIa2} and references therein.  Observe that  the same nonlinear equations also appear in prescribed Lorentzian mean
	curvature problems, e.g. \cite{BIM,BIa2,Bartnik}.
	
	In the present paper, we consider $\cL_q$ in the electromagnetostatic case with $q\in [1,2)$.
	
	First, we are interested in finding  the electrostatic potential $\phi$
	in the presence of a fixed magnetic field $B=B(x)$. For a given {\em charge density} $\rho$, the corresponding Euler-Lagrange equation, at least formally, is
	\begin{equation}\label{eqs}
		-\div\left(\frac{[\nabla\phi+(\nabla\phi\cdot B)B]}{\big(1+|B|^2-|\nabla \phi|^2-(\nabla \phi\cdot B)^2\big)^{1-q/2}}\right)=\rho
		\quad\text{in }\R^3.
	\end{equation}
	Moreover
	\eqref{eqs} can be studied by means of the action functional
	\begin{equation}\label{eq:functionalBI}
		I_B(\phi)
		=
		\frac{1}{q} \int_{\R^3} \Big(1 - \big(1+|B|^2-|\nabla \phi|^2-(\nabla \phi\cdot B)^2\big)^{q/2}\Big) dx
		- \langle \rho, \phi\rangle.
	\end{equation}
	As we shall see in Section \ref{sec:electrostatic}, $I_B$ is well-defined on a closed and convex subset $X_B$ of $\D^{1,2}(\R^3)$. Following  \cite{DenisPietroAlessio}, where $B=0$ and $q=1$, we show that $I_B$ attains its minimum $\phi_\rho$. However it is not clear if it solves \eqref{eqs} in a suitable sense. Then, considering $B\neq 0$, we assume in addition that
	\begin{equation}\label{eq:cylassum}
		B(x_1,x_2,x_3)=\frac{b(x_1,x_2,x_3)}{\sqrt{x_1^2+x_2^2}}(-x_2,x_1,0),
	\end{equation}
	where $b:\R^3\to\R$ is {\em radially symmetric}, i.e.  invariant with respect to the orthogonal group action $\cO(3)$. 
	
	If we denote by $X_B^*$ the {\em dual} of $X_B$, our first main result reads as follows.
	
	\begin{Th}\label{thm:radial}
		Let $b\in L^2(\R^3)\cap L^\infty(\R^3)$ be radially symmetric, $\rho\in X_B^*$, $\rho\neq 0$ be a radial distribution of charge, and $q\in [1,2)$. Then there is a unique and nontrivial minimizer $\phi_\rho$ of $I_B$, which is a weak and radial solution to the electrostatic problem \eqref{eqs}.
	\end{Th}

	Note that the magnetic field $B\neq 0$ of the form \eqref{eq:cylassum} is not $\cO(3)$-equivariant. Indeed, if $B$ is $\cO(3)$-equivariant and $B=\curlop A$ for some $A\in W^{1,1}_{loc}(\R^3,\R^3)$, then $B=0$ (see Proposition \ref{prop:B0}).
	Recall that Theorem \ref{thm:radial} for $B=0$ and $q=1$ has been obtained  in \cite{DenisPietroAlessio}.
	
	Next we are interested in finding magnetic potential $A$ when $\phi=0$. This leads to
	the following equation
	\begin{equation}\label{eq}
		\curlop\left(\displaystyle\frac{\curlop A}{\big(1+|\curlop A|^2\big)^{1-q/2}}\right)=J 
		\quad\text{in }\R^3,
	\end{equation}
	where $J$ is an {\em external current density}.
	
	If $J=0$ and $q=1$, Yang \cite{Yang} showed that there are no nontrivial solutions to \eqref{eq}, and the natural open question arose concerning  nontrivial solutions in the presence of nontrivial current $J$. We answer to this problem for $\cL_q$ with $q\in (6/5,2)$. 
	
	When $q=2$ and $J$ depends nonlinearly on $A$ then \eqref{eq}  was recently investigated e.g. in \cite{BenForAzzAprile,BenFor,MedSz,Mederski}, where \eqref{eq}  with $q=2$ was motivated by the search of  the exact propagation of electromagnetic waves in nonlinear media arising in optics. 
	Furthermore, a similar (and simpler) nonlinear operator as in \eqref{eq} in the scalar case (and so involving the divergence and the gradient instead the $\curlop$ operator) with a nonlinear right hand side has been studied in \cite{ADP}. To the best of our knowledge we present the first analytical study of the existence of solutions to \eqref{eq} with the fixed nontrivial external current source $J$. As we shall see in Theorem \ref{th:symmetry} we cannot expect radial solutions $A$, i.e. $\cO(3)$-equivariant.

	In order to treat \eqref{eq} variationally we must take into account the first difficulty concerning the operator 
	$$A\mapsto\curlop\left(\frac{\curlop A}{\big(1+|\curlop A|^2\big)^{1-q/2}}\right),$$
	which disappears on the space of gradient fields $A=\nabla\phi$. Then the natural energy functional given by
	\begin{equation*}
		\cJ(A):=\frac{1}{q}\int_{\R^3} \big[\big(1+|\curlop A|^2)^{q/2}-1\big]\, dx- \langle J,A\rangle
	\end{equation*}
	is trivial on the gradient fields $A=\nabla\phi$, $\phi\in\cC_0^{\infty}(\R^3)$ and, for suitable $J$'s, $\langle J,A\rangle=0$ since $\div J=0$.
	Therefore, due to the gauge invariance of \eqref{eq} it is natural to look for divergence-free solutions, i.e. $\div A=0$. 
	
	In Section \ref{sec:magnetic} we define a Banach space $\cA$ of divergence-free vector field  in which $\cJ$ is well defined and
	we obtain the following result.
	\begin{Th}\label{th:mainBI}
		Let $J\in\cA^*$, $J\neq 0$ and $q\in (6/5,2)$. Then
		there is a nontrivial and cylindrically symmetric (weak) solution to \eqref{eq} of the form 
		$$A(x_1,x_2,x_3)=\frac{u(r,x_3)}{r}(-x_2,x_1,0)\hbox{ with }u:(0,\infty)\times\R\to\R,\; r=\sqrt{x_1^2+x_2^2}$$
		such that $A$ is the global minimizer of $\J$ in $\A$.
	\end{Th}
	
	The construction of space $\cA$ and the use of variational approach requires $q>6/5$ (see Section \ref{sec:magnetic}). The problem for $q\in [1,6/5]$ remains open.\\
	Observe that, if $q=2$, both in \eqref{eqs} and \eqref{eq} we recover {\em classical} operators and standard methods can be applied.
	
	From now on we assume that $q\in [1,2]$. In what follows,  $| \cdot |_k$ denotes the $L^k$-norm  for $k\in [1,+\infty]$.
	Moreover, with $C,C_i$ we denote positive constants that can vary also from line to line.

	\section{Electrostatic case}\label{sec:electrostatic}
	
	In this section we will study the electrostatic case in presence of a magnetic field $B=B(x)$ for $q\in[1,2)$.
	
	Let us fix the magnetic field $B\in L^2(\R^3,\R^3)\cap L^\infty(\R^3,\R^3)$
	and consider the set
	\begin{equation*}
		X_B
		:=
		\cD^{1,2}(\R^3)
		\cap
		\big\{\phi\in \cC^{0,1}(\R^3) :\;
		|\nabla \phi|^2 +(\nabla \phi\cdot B)^2\le 1+|B|^2\hbox{ a.e. in }\R^3\big\}
	\end{equation*}
	equipped with the norm
	\begin{equation*}
		\|\phi\|_{X_B}:=\left(\int_{\R^3}|\nabla\phi|^2 \ dx\right)^{1/2},
	\end{equation*}
	where $\cD^{1,2}(\R^3)$ is the completion of $\cC_0^\infty (\R^3)$ with respect to $\|\cdot\|_{X_B}$.\\
	Observe that $X_B$ is a convex subset of $\cD^{1,2}(\R^3)$. Moreover, we have
	\begin{Lem} $X_B$ is weakly closed subset of $\cD^{1,2}(\R^3)$.
	\end{Lem}
	\begin{proof}
		Let us take a sequence $\{ \phi_n \}\subset X_B$ such that $\phi_n\to \phi$ in $\cD^{1,2}(\R^3)$. Clearly  $\phi_n\to \phi$ a.e. on $\R^3$ passing to a subsequence, and since $\phi_n$'s are Lipschitz continuous, for every $n\in\N$ there is a constant $c_n>0$ such that
		$$|\phi_n(x)-\phi_n(y)|\leq c_n |x-y|,\quad\hbox{for }x,y\in\R^3.$$
		Moreover, since $\{ \phi_n \}\subset X_B$, then
		$|\nabla\phi_n|_\infty\leq 1+|B|_\infty$ and so we may assume that $\{c_n\}$ is bounded.
		Now, the boundedness of $\{\nabla\phi_n\}$ in $L^2(\R^3,\R^3)$ and in $L^\infty(\R^3,\R^3)$ implies, by the Sobolev embeddings, that $\{\phi_n\}$ is bounded in $L^\infty(\R^3)$ and passing to a subsequence 
		$\phi_n\to \phi$ in $L^\infty(\R^3)$. Therefore,  for some constant $c>0$,
		$$|\phi(x)-\phi(y)|\leq c |x-y|,\quad\hbox{for }x,y\in\R^3.$$
		Then we conclude that $\phi\in X_B$, and $X_B$ is closed.  Hence $X_B$ is weakly closed, since it convex.
	\end{proof}

	The following fundamental inequalities hold.
	
	\begin{Lem}
		Let $\phi\in X_B$. For a.e. $x\in\R^3$
		\begin{equation}\label{fundest}
			1-|\nabla \phi(x)|^2
			\leq 
			\big(1+|B(x)|^2-|\nabla \phi(x)|^2-(\nabla \phi(x)\cdot B(x))^2\big)^{q/2}
			\leq
			1+\frac{q}{2}(|B(x)|^2-|\nabla \phi(x)|^2).
		\end{equation}
	\end{Lem}
	\begin{proof}
		Let us start proving the left inequality.\\
		If $|\nabla\phi(x)|\geq 1$, then
		\[
		\big(1+|B(x)|^2-|\nabla \phi(x)|^2-(\nabla \phi(x)\cdot B(x))^2\big)^{q/2}
		\geq 0
		\geq
		1-|\nabla \phi(x)|^2.
		\]
		If, instead, $|\nabla\phi(x)|< 1$, then
		\[
		(1-|\nabla \phi(x)|^2)^{1-q/2}
		< 1
		\leq 
		(1+|B(x)|^2)^{q/2}
		\]
		and so
		\[
		1-|\nabla \phi(x)|^2
		<
		(1+|B(x)|^2)^{q/2}(1-|\nabla \phi(x)|^2)^{q/2}.
		\]
		Thus
		\[
		1+|B(x)|^2-|\nabla \phi(x)|^2-|\nabla \phi(x)|^2 |B(x)|^2
		=\big(1+|B(x)|^2\big) \big(1-|\nabla \phi(x)|^2\big)
		>0
		\]
		and
		\begin{align*}
			\big(1+|B(x)|^2-|\nabla \phi(x)|^2-(\nabla \phi(x)\cdot B(x))^2\big)^{q/2}
			&\geq
			\big(1+|B(x)|^2-|\nabla \phi(x)|^2-|\nabla \phi(x)|^2 |B(x)|^2\big)^{q/2}\\
			&=\big(1+|B(x)|^2\big)^{q/2}
			\big(1-|\nabla \phi(x)|^2\big)^{q/2}\\
			&> 1-|\nabla \phi(x)|^2.
		\end{align*}
		Now let us prove the second inequality.\\
		Observe that, by the definition of $X_B$,
		\[
		1+\frac{q}{2}(|B(x)|^2-|\nabla \phi(x)|^2)
		\geq
		1+\frac{q}{2}((\nabla\phi (x)\cdot B(x))^2-1)
		>
		0
		\]
		and
		\[
		1+|B(x)|^2-|\nabla \phi(x)|^2
		\geq
		0.
		\]
		Then
		\begin{align*}
			\big(1+|B(x)|^2-|\nabla \phi(x)|^2-(\nabla \phi(x)\cdot B(x))^2\big)^{q}
			&\leq
			\big(1+|B(x)|^2-|\nabla \phi(x)|^2\big)^{q}\\
			&\leq
			\Big[1+\frac{q}{2}(|B(x)|^2-|\nabla \phi(x)|^2)\Big]^{2}
		\end{align*}
		where the last inequality follows observing that the function $t\in [-1,+\infty)\mapsto (1+t)^q-(1+qt/2)^2$ is nonpositive.
	\end{proof}
	
	Let $\rho\in X^*_B$. Then  $\langle\ , \ \rangle$ in \eqref{eq:functionalBI} denotes the duality pairing between $X_B^*$ and $X_B$.	
	As an immediate consequence of the previous Lemma, we have that the functional $I_B:X_B\to\R$ is well-defined in $X_B$.
	
	Moreover if $\phi,\psi\in X_B$ and $|\nabla \phi|_\infty<1$ and $|\nabla \psi|_\infty<1$, then $I_B'(\phi)[\psi]$ exists and
	\[
	I_B'(\phi)[\psi]
	=
	\int_{\R^3} \frac{[\nabla\phi+(\nabla\phi\cdot B)B]\cdot\nabla\psi}{\big(1+|B|^2-|\nabla \phi|^2-(\nabla \phi\cdot B)^2\big)^{1-q/2}}\,dx
	- \langle \rho,\psi\rangle.
	\]
	
	Observe that, in such a case, for a.e. $x\in\R^3$,
	\[
	1+|B(x)|^2-|\nabla \phi(x)|^2-(\nabla \phi(x)\cdot B(x))^2
	\geq(1+|B(x)|^2)(1-|\nabla \phi|_\infty^2)
	\geq
	1-|\nabla \phi|_\infty^2>0
	\]
	and
	\begin{equation}\label{finint}
		\left|\frac{[\nabla\phi+(\nabla\phi\cdot B)B]\cdot\nabla\psi}{\big(1+|B|^2-|\nabla \phi|^2-(\nabla \phi\cdot B)^2\big)^{1-q/2}}\right|
		\leq
		\frac{(1+|B|_\infty^2)^{q/2}}{(1-|\nabla \phi|_\infty^2)^{1-q/2}}|\nabla \phi||\nabla\psi| \in L^1(\R^3).
	\end{equation}

	Therefore we give the following definition.
	\begin{Def}\label{def_ws}
		A {\em weak solution} of the electrostatic problem \eqref{eqs} is a function $\phi_\rho\in X_B$ such that for all $\psi \in X_B$, we have
		\[
		\int_{\R^3} \frac{[\nabla\phi_\rho+(\nabla\phi_\rho\cdot B)B]\cdot\nabla\psi}{\big(1+|B|^2-|\nabla \phi_\rho|^2-(\nabla \phi_\rho\cdot B)^2\big)^{1-q/2}}\,dx= \langle \rho,\psi\rangle.
		\]
	\end{Def}
	
	Thus, at least formally, critical points of $I_B$ in $X_B$ are solutions of \eqref{eqs}.

	\begin{Prop}\label{PropIB}
		The functional $I_B$ is bounded from below, coercive, continuous, strictly convex, and weakly lower semi-continuous.
	\end{Prop}
	\begin{proof}
		The boundedness from below and the coercivity are immediate consequences of \eqref{fundest}.\\
		The continuity can be obtained observing that, if the sequence $\{\phi_n\}\subset X_B$ converges to $\phi$ in $X_B$, then, up to a subsequence, $\nabla\phi_n \to \nabla\phi$ a.e. in $\R^3$ and there exists $w\in L^1(\R^3)$ such that $|\nabla\phi_n|^2,|\nabla\phi|^2\leq w$ a.e. in $\R^3$. Thus, by \eqref{fundest},
		\begin{align*}
			&\left|\Big[1-\big(1+|B|^2-|\nabla \phi_n|^2-(\nabla \phi_n\cdot B^2\big)^{q/2}\Big]
			-\Big[1-\big(1+|B|^2-|\nabla \phi|^2-(\nabla \phi\cdot B^2\big)^{q/2}\Big]\right|\\
			&\quad \leq
			\max\left\{ \frac{q}{2} \big| |\nabla \phi_n|^2 - |B|^2 \big|, |\nabla \phi_n|^2\right\}
			+ \max\left\{ \frac{q}{2} \big| |\nabla \phi|^2 - |B|^2 \big|, |\nabla \phi|^2\right\}\\
			& \leq
			2\max\left\{ \frac{q}{2} (w+ |B|^2), w\right\} \in L^1(\R^3)
		\end{align*}
		and then, by Lebesgue Dominated Convergence Theorem we can conclude. It is straightforward to check that the convergence holds for the whole sequence.\\
		About the strict convexity we observe that the function
		\[
		X=(X_1,X_2,X_3)\in\R^3\mapsto 1+|B(x)|^2-|X|^2-(X\cdot B(x))^2
		\]
		is strictly concave. Then its composition with the power function $q/2$, which is also strictly concave being $q\in[1,2)$, allows us to conclude.\\
		Finally, the weakly lower semi-continuity is a consequence of the continuity and the strict convexity of $I_B$.
	\end{proof}

	Thus we obtain the following result.
	\begin{Th}
		There exists a unique  minimiser $\phi_\rho$ of $I_B$ in $X_B$. If $\rho\neq 0$, then $\phi_\rho\neq 0$.
	\end{Th}
	\begin{proof}
		Existence and uniqueness of the minimiser are consequences of Proposition \ref{PropIB}.\\
		To show that such a minimiser is nontrivial, observe that, if $t>0$ is small enough and $\phi\in X_B\setminus\{0\}$ with $|\nabla\phi|_\infty<1$ and  $\langle \rho,\phi\rangle>0$,
		\begin{align*}
			I_B(t\phi)-I_B(0)
			&=
			\frac{1}{q} \int_{\R^3} \Big(\big(1+|B|^2\big)^{q/2} - \big(1+|B|^2-t^2|\nabla \phi|^2-t^2(\nabla \phi\cdot B)^2\big)^{q/2}\Big) dx 
			-t\langle \rho, \phi\rangle.
		\end{align*}
		First observe that, for a.e. $x\in\R^3$,
		\[
		|\nabla \phi(x)|^2+(\nabla \phi(x)\cdot B(x))^2
		\leq
		|\nabla\phi|_\infty^2 (1+|B(x)|^2).
		\]
		Now, let $\mathfrak{K}_1>1$, $\mathfrak{K}_2\geq0$, $\mathfrak{K}_2\leq\kappa\mathfrak{K}_1$ with $\kappa\in(0,1)$, and, for $t\in(0,1)$, consider the function $g(s):=(\mathfrak{K}_1-s^2\mathfrak{K}_2)^{q/2}$ in $[0,t]$. If $\mathfrak{K}_2=0$, then $g$ is constant so that $g(t)-g(0)=0$. If $\mathfrak{K}_2>0$, applying the Lagrange Theorem we have that there exists $\xi\in[0,t]$ such that
		\[
		g(t)-g(0)=g'(\xi)t,
		\]
		namely
		\[
		\mathfrak{K}_1^{q/2} - (\mathfrak{K}_1-t^2\mathfrak{K}_2)^{q/2} 
		=\frac{q\mathfrak{K}_2}{(\mathfrak{K}_1-\xi^2\mathfrak{K}_2)^{1-q/2}} \xi t
		\leq \frac{q\mathfrak{K}_2}{(\mathfrak{K}_1-\mathfrak{K}_2)^{1-q/2}} t^2,
		\]
		being
		$$
		\mathfrak{K}_1-\xi^2\mathfrak{K}_2
		\geq\mathfrak{K}_1-t^2\mathfrak{K}_2
		\geq\mathfrak{K}_1-\mathfrak{K}_2
		\geq(1-\kappa)\mathfrak{K}_1>1-\kappa>0.$$
		Hence, applying the previous arguments for $\mathfrak{K}_1:=1+|B(x)|^2$, $\mathfrak{K}_2:=|\nabla \phi(x)|^2+(\nabla \phi(x)\cdot B(x))^2$, $x\in\R^3$, and $\kappa=|\nabla\phi|_\infty$,
		we obtain
		\[
		I_B(t\phi)-I_B(0) 
		\leq t^2 \int_{\R^3} \frac{|\nabla \phi|^2+(\nabla \phi\cdot B)^2}{\big(1+|B|^2-|\nabla \phi|^2-(\nabla \phi\cdot B)^2\big)^{1-q/2}} dx 
		-t\langle \rho, \phi\rangle
		\]
		and so, observing that by \eqref{finint} for $\psi=\phi$ the integral in the previous formula is finite 
		and taking $t>0$ small enough, we can conclude.
	\end{proof}

	The importance of the minimiser of $I_B$ relies in the fact that, due to the convexity and using Definition \ref{def_ws}, it can be proved in a classical way (analogously to \cite[Proposition 2.6]{DenisPietroAlessio}) that a weak solution of \eqref{eqs} must minimise $I_B$.
	
	Moreover, such a minimiser satisfies the following property.
	
	\begin{Prop}\label{pr:sol-deb}
		Assume $\rho\in X^{*}_B$ and let $\phi_\rho$ be the unique minimizer of $I_B$ in $X_B$. Then for all $\psi\in X_B $, we have the variational inequality
		\begin{equation}	\label{eq:ineqcc}
			\int_{\R^3} \frac{  |\nabla  \phi_\rho|^2+(\nabla\phi_\rho\cdot B)^2- \nabla \phi_\rho \cdot \nabla \psi-(\nabla\phi_\rho\cdot B)(\nabla\psi\cdot B)}{ \big(1+|B|^2- |\nabla \phi_\rho|^2-(\nabla\phi_\rho\cdot B)^2\big)^{1-q/2}}\, dx
			\le \langle \rho,\phi_\rho\rangle - \langle\rho,\psi \rangle.
		\end{equation}
	\end{Prop}
	\begin{proof}
		Observe that for every $t\in [0,1]$ and $\psi\in X_B$, $\phi_t=\phi_\rho +t (\psi - \phi_\rho)\in X_B$, we have $I_B(\phi_\rho)\le I_B(\phi_t)$, namely
		\begin{align*}
			\xi(\psi,t)&:=\frac1q\int_{\R^3} \displaystyle\Big(1+|B|^2-|\nabla\phi_{t}|^2 -(\nabla\phi_t\cdot B)^2\Big)^{q/2}-\Big(1+|B|^2-|\nabla\phi_\rho|^2-(\nabla\phi_\rho\cdot B)^2\Big)^{q/2}\,  dx\\
			&\le t (\langle \rho,\phi_\rho\rangle - \langle\rho,\psi \rangle).
		\end{align*}
		Observe that for $t\in(0,1)$
		$$ \frac{ \xi(\psi,t)- \xi(\psi,0)}{t}\leq \langle \rho,\phi_\rho\rangle - \langle\rho,\psi \rangle
		$$
		and so
		\begin{equation}\label{eq:ineqlimsup}
			\limsup_{t\to 0^+}\frac{ \xi(\psi,t)- \xi(\psi,0)}{t}\leq  \langle \rho,\phi_\rho\rangle - \langle\rho,\psi \rangle.
		\end{equation}
		Moreover, for $\psi=0$, we get $ \phi_t=(1-t) \phi_\rho$ and so 
		$$\Big(1+|B|^2-|\nabla\phi_{t}|^2 -(\nabla\phi_t\cdot B)^2\Big)^{q/2}-\Big(1+|B|^2-|\nabla\phi_\rho|^2-(\nabla\phi_\rho\cdot B)^2\Big)^{q/2}\geq 0
		\quad \text{a.e. in }\R^3$$
		and, if $E_B:=\{x\in\R^3:1+|B(x)|^2-|\nabla\phi_\rho(x)|^2-(\nabla\phi_\rho(x)\cdot B(x))^2=0\}$,
		\begin{align*}
			&\langle \rho,\phi_\rho \rangle\\
			&\geq \frac{ \xi(\psi,t)- \xi(\psi,0)}{t}\\
			&=
			\int_{E_B^c} \frac{\Big(1+|B|^2-|\nabla\phi_{t}|^2 -(\nabla\phi_t\cdot B)^2\Big)^{q/2}-\Big(1+|B|^2-|\nabla\phi_\rho|^2-(\nabla\phi_\rho\cdot B)^2\Big)^{q/2}}{t} \, dx\\
			&\qquad
			+ \frac{(2-t)^{q/2}}{t^{1-q/2}}	\int_{E_B} (1+|B|^2)^{q/2}\, dx\\
			&\geq
			\int_{E_B^c} \frac{\Big(1+|B|^2-|\nabla\phi_{t}|^2 -(\nabla\phi_t\cdot B)^2\Big)^{q/2}-\Big(1+|B|^2-|\nabla\phi_\rho|^2-(\nabla\phi_\rho\cdot B)^2\Big)^{q/2}}{t} \, dx\\
			&\qquad
			+ \frac{(2-t)^{q/2}}{t^{1-q/2}}	|E_B|.
		\end{align*}
		Therefore, $|E_B|=0$ and, by Fatou's lemma,
		\begin{equation}\label{eq:inqrhopsi}
			\int_{\R^3}\frac{  |\nabla  \phi_\rho|^2+(\nabla\phi_\rho\cdot B)^2}{ \big(1+|B|^2- |\nabla \phi_\rho|^2-(\nabla\phi_\rho\cdot B)^2\big)^{1-q/2}}\,dx
			\leq
			\limsup_{t\to 0^+}\frac{ \xi(0,t)- \xi(0,0)}{t}
			\leq  \langle \rho,\phi_\rho \rangle,
		\end{equation}
		hence \eqref{eq:ineqcc} holds for $\psi=0$.\\
		Let us assume that $\psi\neq 0$ and write
		\[
		\frac{ \xi(\psi,t)- \xi(\psi,0)}{t}
		= \int_{\R^3} \frac{\mathfrak{f}(\psi,t)-\mathfrak{f}(\psi,0)}{t} \, dx
		\]
		where
		\[
		\mathfrak{f}(\psi,t)=\frac{1}{q}\Big(1+|B|^2-|\nabla\phi_{t}|^2 -(\nabla\phi_t\cdot B)^2\Big)^{q/2}.
		\]
		We claim that
		\[
		\left|\frac{ \mathfrak{f}(\psi,t)- \mathfrak{f}(\psi,0)}{t}\right| \\
		\le 
		C \frac{ |\nabla \phi_\rho |^2+ (\nabla\phi_\rho\cdot B)^2+  |\nabla \psi|^2+ (\nabla\psi\cdot B)^2}{  \big(1+|B|^2- |\nabla \phi_\rho|^2-(\nabla\phi_\rho\cdot B)^2\big)^{1-q/2}}\in L^1(\R^3).
		\]
		Using the Lagrange Theorem we have
		\begin{equation}\label{DCT}
			\frac{\mathfrak{f}(\psi,t)-\mathfrak{f}(\psi,0)}{t}
			=\frac{\partial\mathfrak{f}}{\partial t}(\psi,\vartheta)
			=-\frac{\nabla\phi_{\vartheta} \cdot \nabla(\psi-\phi_\rho)+(\nabla\phi_{\vartheta} \cdot B)(\nabla(\psi-\phi_\rho)\cdot B)}{\Big(1+|B|^2-|\nabla\phi_{\vartheta}|^2 -(\nabla\phi_{\vartheta}\cdot B)^2\Big)^{1-q/2}}
		\end{equation}
		with $\vartheta\in[0,t]$.\\
		Observe that, since
		\begin{align*}
			\left|\nabla\phi_\vartheta \cdot \nabla(\psi-\phi_\rho)\right|
			&\leq
			(1-\vartheta)|\nabla\phi_\rho| |\nabla\psi|
			+(1-\vartheta)|\nabla\phi_\rho|^2
			+\vartheta|\nabla\psi|^2
			+\vartheta|\nabla\phi_\rho|
			|\nabla\psi|\\
			&=
			(1-\vartheta)|\nabla\phi_\rho|^2
			+|\nabla\phi_\rho|
			|\nabla\psi|
			+\vartheta|\nabla\psi|^2\\
			&\leq
			|\nabla\phi_\rho|^2
			+|\nabla\phi_\rho|
			|\nabla\psi|
			+|\nabla\psi|^2
			\leq
			2(|\nabla\phi_\rho|^2
			+|\nabla\psi|^2)
		\end{align*}
		and, analogously,
		\[
		\left|(\nabla\phi_\vartheta \cdot B)(\nabla(\psi-\phi_\rho)\cdot B)\right|
		\leq
		2[(\nabla\phi_\rho\cdot B)^2
		+(\nabla\psi\cdot B)^2],
		\]
		we have
		\[
		\left|\nabla\phi_\vartheta \cdot \nabla(\psi-\phi_\rho)+(\nabla\phi_\vartheta \cdot B)(\nabla(\psi-\phi_\rho)\cdot B)\right|
		\leq
		2[|\nabla\phi_\rho|^2
		+(\nabla\phi_\rho\cdot B)^2
		+|\nabla\psi|^2
		+(\nabla\psi\cdot B)^2].
		\]
		Moreover, since
		\begin{align*}
			|\nabla\phi_{\vartheta}|^2
			&
			=(1-\vartheta)^2|\nabla\phi_{\rho}|^2
			+2 \vartheta(1-\vartheta)\nabla\phi_{\rho}\cdot\nabla\psi
			+ \vartheta^2|\nabla\psi|^2\\
			&\leq
			(1-\vartheta)^2|\nabla\phi_{\rho}|^2
			+ \vartheta(1-\vartheta)|\nabla\phi_{\rho}|^2
			+ \vartheta(1-\vartheta)|\nabla\psi|^2
			+ \vartheta^2|\nabla\psi|^2\\
			&=
			(1-\vartheta)|\nabla\phi_{\rho}|^2
			+ \vartheta|\nabla\psi|^2
		\end{align*}
		and, analogously,
		\[
		(\nabla\phi_{\vartheta}\cdot B)^2
		\leq
		(1-\vartheta)(\nabla\phi_{\rho}\cdot B)^2
		+ \vartheta(\nabla\psi\cdot B)^2,
		\]
		since $\psi\in X_B$, we have that
		\begin{align*}
			1+|B|^2-|\nabla\phi_{\vartheta}|^2 -(\nabla\phi_{\vartheta}\cdot B)^2
			&\geq
			1+|B|^2
			-(1-\vartheta)[|\nabla\phi_{\rho}|^2+(\nabla\phi_{\rho}\cdot B)^2]\\
			&\quad
			- \vartheta[|\nabla\psi|^2+(\nabla\psi\cdot B)^2]\\
			&\geq
			(1-\vartheta) [1+|B|^2
			-|\nabla\phi_{\rho}|^2-(\nabla\phi_{\rho}\cdot B)^2]\\
			&\geq
			(1-t) [1+|B|^2
			-|\nabla\phi_{\rho}|^2-(\nabla\phi_{\rho}\cdot B)^2]
		\end{align*}
		Then
		\[
		\left|\frac{\mathfrak{f}(\psi,t)-\mathfrak{f}(\psi,0)}{t}
		\right|
		\leq
		\frac{2}{(1-t)^{1-q/2}}
		\frac{ |\nabla \phi_\rho |^2+ (\nabla\phi_\rho\cdot B)^2+  |\nabla \psi|^2+ (\nabla\psi\cdot B)^2}{  \big(1+|B|^2- |\nabla \phi_\rho|^2-(\nabla\phi_\rho\cdot B)^2\big)^{1-q/2}}.
		\]
		From \eqref{eq:inqrhopsi} we infer that 
		$$\frac{ |\nabla \phi_\rho |^2+ (\nabla\phi_\rho\cdot B)^2}{  \big(1+|B|^2- |\nabla \phi_\rho|^2-(\nabla\phi_\rho\cdot B)^2\big)^{1-q/2}}\in L^1(\R^3).$$
		Moreover, if $\cA_\rho:=\{x\in\R^3:|\nabla \phi_\rho(x)|^2+(\nabla\phi_\rho(x)\cdot B(x))^2 < (1+|B(x)|^2)/4 \}$,
		since in $\cA_\rho^c$
		\[
		|\nabla \psi|^2 + (\nabla\psi\cdot B)^2
		\leq 1+|B|^2
		\leq 4 ( |\nabla \phi_\rho|^2 +(\nabla\phi_\rho\cdot B)^2),
		\]
		we have
		\begin{align*}
			&\int_{\R^3} \frac{ |\nabla \psi|^2+ (\nabla\psi\cdot B)^2 }{  \big(1+|B|^2- |\nabla \phi_\rho|^2-(\nabla\phi_\rho\cdot B)^2\big)^{1-q/2}}\, dx\\
			&=
			\int_{\cA_\rho} \frac{ |\nabla \psi|^2 + (\nabla\psi\cdot B)^2}{  \big(1+|B|^2- |\nabla \phi_\rho|^2-(\nabla\phi_\rho\cdot B)^2\big)^{1-q/2}}\, dx\\
			&\qquad
			+\int_{\cA_\rho^c} \frac{ |\nabla \psi|^2 + (\nabla\psi\cdot B)^2}{  \big(1+|B|^2- |\nabla \phi_\rho|^2-(\nabla\phi_\rho\cdot B)^2\big)^{1-q/2}}\, dx\\
			&\leq
			C\Big(\int_{\R^3} |\nabla \psi|^2+ (\nabla\psi\cdot B)^2 \, dx
			+\int_{\R^3} \frac{ |\nabla \phi_\rho|^2 +(\nabla\phi_\rho\cdot B)^2}{  \big(1+|B|^2- |\nabla \phi_\rho|^2-(\nabla\phi_\rho\cdot B)^2\big)^{1-q/2}}\, dx\Big).
		\end{align*}
		In view of \eqref{DCT}, by Lebesgue's Dominated Convergence Theorem, we may compute \eqref{eq:ineqlimsup} which implies \eqref{eq:ineqcc}.
	\end{proof}

	\subsection{Cylindrical magnetic field}

	We show that, in general,  we cannot consider $\cO(3)$-equivariant $B\neq 0$. 
	Indeed we have
	\begin{Prop}\label{prop:B0}
		If $B$ is $\cO(3)$-equivariant and $B=\curlop A$ for some $A\in W^{1,1}_{loc}(\R^3,\R^3)$, then $B=0$.
	\end{Prop}
	
	To prove it we need the following preliminary result.
	\begin{Lem}\label{lem:O3grad}
		If $A\in L^1_{loc}(\R^3,\R^3)$ is $\cO(3)$-equivariant, then 	$A(x)=\nabla \psi(|x|)$ for $x\in \R^3\setminus\{0\}$ for some absolutely continuous function $\psi:(0,+\infty)\to\R$.
	\end{Lem}
	\begin{proof}
		Let us fix $x\in\R^3\setminus\{0\}$ and consider the isotropy group of $x$
		$$\cO_x=\{g\in \cO(3)|\; gx=x\}.$$
		Observe that
		$$\{y\in\R^3\setminus\{0\}|\; \cO_x=\cO_y\}=\R x\setminus\{0\}.$$
		Then, if $x\in\R^3\setminus\{0\}$, $gA(x)=A(gx)=A(x)$ for all $g\in\cO_x$ and so $A(x)\in\R x\setminus\{0\}$. Hence, using the $\cO(3)$-invariance of $A$, there exists a function $\vp:(0,+\infty)\to\R\setminus\{0\}$ such that 			$A(x)=\vp(|x|)x/|x|$ and, due to the local integrability of $A$,
		the map 
		$$\R^3\setminus\{0\}\ni x\mapsto\vp(|x|)= A(x)\cdot \frac{x}{|x|}\in \R\setminus\{0\}$$
		is locally integrable.\\
		Thus, let us consider the function $\psi:(0,+\infty)\to\R$, defined by $$\psi(r)=\int_1^r \vp(s)\, ds.$$ 
		Then we have that $\psi$ is absolutely continuous and $A(x)=\nabla \psi(|x|)$ for $x\in \R^3\setminus\{0\}$.
	\end{proof}
	
	\begin{proof}[Proof of Proposition \ref{prop:B0}]
		Let  $B$ be $\cO(3)$-equivariant. By Lemma \ref{lem:O3grad}, $B(x)=\nabla \psi (|x|)$ for some absolutely continuous function $\psi:(0,\infty)\to \R$. Since $\curlop B=0$ in the distributional sense, in view of \cite[Lemma 1.1]{Le}, there is $\vp\in W^{1,1}_{loc}(\R^3,\R^3)$ such that $B=\nabla \vp$. Since $\div(B)=\div(\curlop A)$ in the distributional sense, we get $\div(\nabla \vp)=-\Delta \vp =0$, hence $\vp$ is a harmonic function. Therefore $B=\nabla \vp$ is harmonic as well, and since $B\in L^2(\R^3,\R^3)$, we obtain $B=0$.
	\end{proof}

	In view of Proposition \ref{prop:B0} the $\cO(3)$-equivariance of $B$ is too strong. Therefore we assume that $B$ is of the form \eqref{eq:cylassum} with $b$ radially symmetric so that $B$ is $\cO(2)\times\id$-equivariant.\\
	We observe that 	
	\begin{equation}\label{orth}
		\nabla \phi\cdot B=0
	\end{equation}
	provided that $\phi\in X_B$ is radial.

	Let $X_r$
	be the subset of radial functions of $X_B$.
	If $\rho\in X_B^*$ is a {\em radial distribution of charge}, namely, for every $g\in \cO(3)$, $g\rho=\rho$, where, for $\phi\in X_B$,  $\langle g\rho,\phi\rangle:=\langle \rho, g\phi\rangle$, being $g\phi(x):=\phi(gx)$, then
	$$I_B(\phi)=\frac{1}{q}\int_{\R^3} \Big(1 - \big(1+|B|^2-|\nabla \phi|^2\big)^{q/2}\Big) dx
	- \langle \rho, \phi\rangle$$
	is $\cO(3)$-invariant.

	\begin{altproof}{Theorem \ref{thm:radial}}
		Since $I_B$ is $\cO(3)$-invariant, then for any $g \in \cO(3)$ we get
		$I_B(g\phi_\rho)=I_B(\phi_\rho)$ and so, $g\phi_\rho=\phi_\rho$ by the uniqueness of the minimum. Therefore  $\phi_\rho \in X_r$ and we will replace $\phi_\rho(x)$ by $\phi_\rho(\tau)$, where $\tau=|x|$. Since $b$ is $\cO(3)$-invariant, then will replace also $b(x)$ by $b(\tau)$.\\
		In order to prove that $\phi_\rho$ is a weak solution of the electrostatic problem, following \cite{ST}, we 
		define
		\[
		E_k:=\left\{\tau\ge 0\;\vline\;|\phi_\rho'(\tau)|^2-|b(\tau)|^2 \ge 1-\frac{1}{k}\right\}
		\]
		for $k\geq 1$.\\
		By \eqref{eq:ineqcc} and \eqref{orth} we infer that
		$$\frac{  |\nabla  \phi_\rho|^2}{ \big(1+|B|^2- |\nabla \phi_\rho|^2\big)^{1-q/2}}\in L^1(\R^3).$$
		Since the numerator $|\nabla \phi_\rho|$ is strictly bounded away from zero on each $E_k$ with $k\geq 2$,
		$E_\infty:=\big\{\tau\ge 0 \mid |\phi_\rho'(\tau)|^2-|b(\tau)|^2=1\big\}$
		is a set of measure $0$. Hence $\left|\bigcap_{k\ge 1} E_k\right|=0$.
		\\
		Now, let us take $\psi\in X_r\cap \cC^\infty_0(\R^3)$ with $\supp \psi\subset [0,R]$ for some $R>0$ and let
		\[
		\psi_k(\tau)=-\int_{\tau}^{+\infty} \psi'(s) [1-\chi_{E_k}(s)] ds.
		\]
		Clearly $\supp \psi_k\subset [0,R]$, for any $k\ge 1$. 
		Observe that for any $t\in\R$
		$$(\phi_\rho+t\psi_k)'(\tau)= \phi_\rho'(\tau)+t\psi'(\tau)[1-\chi_{E_k}(\tau)],$$
		and, if $\tau\in E_k$, then
		\[
		|(\phi_\rho+t\psi_k)'(\tau)|^2=|\phi_\rho'(\tau)|^2\le 1 + |b(\tau)|^2,
		\]
		otherwise,
		$$|(\phi_\rho+t\psi_k)'(\tau)|^2-|b(\tau)|^2\leq  1-\frac1k+2t\phi_\rho'(\tau)\psi'(\tau)+t^2|\psi'(\tau)|^2<1$$
		for $|t|$ small enough.\\
		Therefore $\phi_\rho+t\psi_k\in X_r$ provided that $|t|$ is sufficiently small.\\
		Now, arguing as in the proof of Proposition \ref{pr:sol-deb}, for every $k\ge 1$ we get
		\begin{equation}\label{eq:ELeq}
			\begin{split}
				&\lim_{t\to 0} \frac{I_B(\phi_\rho+t\psi_k)-I_B(\phi_\rho)}{t}\\
				&\qquad=\omega_N\int_0^{+\infty} \frac{\phi_\rho'\psi'}{\big(1+|b|^2-|\phi_\rho'|^2\big)^{1-q/2}}[1-\chi_{E_k}] \tau^{N-1}\, d\tau
				- \langle \rho,\psi_k\rangle = 0.
			\end{split}
		\end{equation}
		Moreover, since $E_{k+1}\subset E_k$ and $|E_k|\to 0$, as $k\to +\infty$, then $\chi_{E_k}\to 0$ a.e. in $\mathbb{R}^N$ and so, using the Lebesgue's Dominated Convergence Theorem, we have
		\[
		\int_0^{+\infty} \frac{\phi_\rho'\psi'}{\big(1+|b|^2-|\phi_\rho'|^2\big)^{1-q/2}}[1-\chi_{E_k}] \tau^{N-1}\, d\tau
		\to
		\int_0^{+\infty} \frac{\phi_\rho'\psi'}{\big(1+|b|^2-|\phi_\rho'|^2\big)^{1-q/2}} \tau^{N-1}\, d\tau.
		\]
		In addiction, due to $\psi_k \to \psi$ in $X_r$ as $k\to +\infty$, we have $\langle \rho,\psi_k\rangle \to \langle \rho,\psi\rangle$.\\
		Hence, taking the limit in \eqref{eq:ELeq} as $k\to\infty$, we conclude that for any $\psi\in X_r\cap \cC_0^\infty(\R^3)$
		\begin{equation}\label{solorad}
			\int_{\R^3} \frac{\nabla \phi_\rho \cdot \nabla \psi}{\big(1+|B|^2-|\nabla \phi_\rho|^2\big)^{1-q/2}} \ dx
			= \langle \rho,\psi\rangle.
		\end{equation}
		Finally, by a density argument we can show that \eqref{solorad} is satisfied also for any $\psi \in X_r$.\\
		Indeed, let $\psi \in X_r$ and take $\psi_n:=\zeta_n \ast (\chi_n \psi)$, where  $\zeta_n$ are smooth radially symmetric mollifiers with compact support, and $\chi_n:\R^3 \to \R$ are smooth radially symmetric functions such that $\chi_n(x)=1$ for $|x|\leq n$ and $\supp \chi_n \subset B(0,2n)$. Then  $\{\psi_n\}_n\subset \cC_0^\infty(\R^3)$, $\psi_n$'s are radially symmetric, $\psi_n\to \psi$ in $\D^{1,2}(\R^3)$, and $\{|\nabla \psi_n|_\infty\}$ is bounded. Then \eqref{solorad} holds for any $\psi \in X_r$.\\
		Now, to prove that \eqref{solorad} holds for every $\psi\in X_B$, we observe that, taking $\phi_\rho$, which is radially symmetric, as test function in \eqref{solorad}, we get
		\[
		\int_{\R^3} \frac{|\nabla \phi_\rho|^2}{\big(1+|B|^2-|\nabla \phi_\rho|^2\big)^{1-q/2}} \ dx
		= \langle \rho,\phi_\rho\rangle.
		\]
		Then, by \eqref{eq:ineqcc}, considering $\psi$ and $-\psi$, we get that for every $\psi\in X_B$,
		\[
		\langle \rho,\psi \rangle
		\le\int_{\R^3} \frac{ \nabla \phi_\rho \cdot \nabla \psi}{ \big(1+|B|^2- |\nabla \phi_\rho|^2\big)^{1-q/2}}\, dx
		\le \langle \rho,\psi \rangle
		\]
		concluding the proof. 
	\end{altproof}

	\section{Magnetostatic fields}\label{sec:magnetic}
	
	Let $J$ be a distribution and, since we are looking for solutions of \eqref{eq} and the curl of any vector field is divergence free, we impose the natural condition $\div J=0$, where the curl and the divergence are understood in the distributional sense.

	First of all we give the definition of solution to \eqref{eq}.
	\begin{Def}
		\label{defws}
		We say that a field $A\in L^1_{loc}(\R^3,\R^3)$ is a (weak) solution to \eqref{eq} if 
		\begin{equation*}%\label{def:sol}
			\int_{\R^3}  \left(\frac{\curlop A}{\big(1+|\curlop A|^2\big)^{1-q/2}}\cdot  \curlop B\right) \,dx =  \langle J,  B\rangle
		\end{equation*}
		for any $B\in\cC_0^{\infty}(\R^3,\R^3)$.
	\end{Def}

	Note that, for \eqref{eq}, one cannot expect {\em radial} solutions $A$, i.e. $\cO(3)$ equivariant vector fields, for $J\neq 0$. Indeed we have
	\begin{Th}\label{th:symmetry}
		Suppose that $A\in L^1_{loc}(\R^3,\R^3)$ is a {\em radially symmetric} solution to \eqref{eq}.
		Then $J=0$ a.e. in $\R^3$. 
	\end{Th}
	\begin{proof}
		If $A$ is $\cO(3)$-equivariant, then by Lemma \ref{lem:O3grad}, $\curlop A=0$ a.e. in $\R^3$ and so, for any $B\in\cC_0^{\infty}(\R^3,\R^3)$,
		\[
		\int_{\R^3}  \left(\frac{\curlop A}{\big(1+|\curlop A|^2\big)^{1-q/2}}\cdot  \curlop B \right)\,dx=0.
		\]
		Since $A$ is a weak solution to \eqref{eq}, then, for all $B\in\cC_0^{\infty}(\R^3,\R^3)$,
		$\langle J,  B\rangle=0$
		and we conclude.
	\end{proof}

	Therefore our aim will be to relax the radial symmetry and we will work in a Banach space $\cA$ that contains cylindrically symmetric vector fields that are solenoidal. 
	We will prove that the functional $\cJ$ is strictly convex and attains its minimum in $\cA$ for $q\in (6/5,2)$.
	
	Due to the different behavior in $0$ and at infinity of the function $x \mapsto (1+x^2)^{q/2}-1$ for $q\in(1,2)$, namely
	\[
	\begin{cases}
		(1+x^2)^{q/2}-1 \approx  qx^2/2&\text{ for } |x| \text{ small,}\\
		(1+x^2)^{q/2}-1 \approx  |x|^q &\text{ for } |x| \text{ large,}
	\end{cases}
	\]
	we consider the following Banach space
	\begin{align*}
		\cL
		&:=L^6(\R^3,\R^3)+ L^{q^*}(\R^3,\R^3)\\
		&=\{A\in\cM(\R^3,\R^3): A=A_1+A_2\;,A_1\in L^6(\R^3,\R^3),A_2\in L^{q^*}(\R^3,\R^3)\},
	\end{align*}
	where $\cM(\R^3,\R^3)$ stands for the space of measurable vector fields in $\R^3$ and $q^*:=3q/(3-q)$.\\
	For any $A\in\cL$ we consider the following norm
	\begin{equation*}
		|A|_{6,q^*}:=\inf\{|A_1|_6 + |A_2|_{q^*}:A=A_1+A_2,\; A_1\in L^6(\R^3,\R^3),A_2\in L^{q^*}(\R^3,\R^3)\}.
	\end{equation*}
	We recall that $\cL$ stands for the Orlicz space with the $N$-function 
	$$t\mapsto \int_0^{|t|}\min\{s^5,s^{q^*-1}\}\,ds,$$
	and, since $q,q^*>1$, $\cL$ is reflexive (see \cite{Rao,BadPisRol}).

	Let now $G := \cS\cO(2)\times{1}\subset \cO(3)$. 
	We can define an action of $G$ on $\cL$ by setting
	\begin{equation}
		\label{actionG}
		(g*A)(x) := g\cdot A(g^{-1}x),
		\quad
		g \in G,\
		A\in \cL.
	\end{equation}
	Let $\cL_G$ be the set of fixed points in $\cL$ with respect to the action \eqref{actionG}, i.e. $A\in\cL_G$ provided that $g\ast A=A$.\\
	In the spirit of \cite{BenForAzzAprile}, we have the following decomposition property. Here the assumption $q>6/5$ plays a crucial role.
	\begin{Prop}\label{prop:decomposition}
		Let $A\in \cL_G$. There is a unique decomposition 
		$$A=A_\tau+A_\rho+A_\zeta$$ with summands of the form
		\begin{equation}\label{eq:summands}
			A_\tau(x)
			= \al(r,x_3)(-x_2,x_1,0),\
			A_\rho(x)
			= \be(r,x_3)(x_1,x_2,0),\
			A_\zeta(x)
			= \ga(r,x_3)(0,0,1),
		\end{equation}
		where $\alpha,\beta,\gamma:(0,+\infty)\times\R\to\R$ such that
		$A_\tau,A_\rho,A_\zeta\in \cL_G$.\\
		Moreover, if, in addition, $q\in(6/5,2)$  and  $\nabla A\in L^1_{loc}(\R^3,\R^3)$, then $\nabla A_\tau, \nabla  A_\rho, \nabla A_\zeta \in L^1_{loc}(\R^3,\R^3)$ and
		\begin{equation}
			\label{curlort}
			\curlop  A_\rho\cdot\curlop A_\tau
			=\curlop A_\tau\cdot\curlop A_\zeta
			= 0
			\quad
			\hbox{a.e. in }\R^3.
		\end{equation}
	\end{Prop}
	\begin{proof}
		Let 
		$$\Sigma:=\left\{(x_1,x_2,x_3)\in \mathbb{R}^3 : x_1=x_2=0 \right\}.$$
		For any $x\in\R^3\setminus\Sigma$, we define $A_\tau(x),A_\rho(x)$ and $A_\zeta(x)$ as projections of the vector $A(x)$ in $\R^3$ along orthogonal directions $(-x_2,x_1,0)$, $(x_1,x_2,0)$ and $(0,0,1)$, so that
		\begin{equation}\label{AAtArAz}
			|A (x)|^2=|A_\tau(x)|^2+|A_\rho(x)|^2+|A_\zeta(x)|^2\quad\hbox{for a.e. }x\in\R^3.
		\end{equation}
		Since $A\in\cL$, so that $A=A_1+A_2$  for some $A_1\in L^6(\R^3,\R^3)$ and $A_2\in L^{q^*}(\R^3,\R^3)$, then, considering the projections of the vector $A_1$ and $A_2$ along the orthogonal directions $(-x_2,x_1,0)$, $(x_1,x_2,0)$ and $(0,0,1)$, and using \eqref{AAtArAz}, we get that $A_\tau,A_\rho,A_\zeta\in \cL$. Moreover, straightforward calculations show that \eqref{eq:summands} holds and that $A_\tau,A_\rho,A_\zeta$ are fixed points for the action \eqref{actionG}.\\
		Suppose now that, in addiction, $\nabla A\in L^1_{loc}(\R^3,\R^3)$. Direct computations show (cf. \cite[Lemma 1]{BenForAzzAprile})
		\begin{equation}\label{eq:nablasquare}
			|\nabla A |^2=|\nabla A_\tau|^2+|\nabla  A_\rho|^2+|\nabla A_\zeta|^2\quad\hbox{ a.e. in }\R^3.
		\end{equation}
		Then $\nabla A_\tau, \nabla  A_\rho, \nabla A_\zeta \in L^1_{loc}(\R^3\setminus\Sigma,\R^3)$, however, it is not immediately obvious that they belong to $L^1_{\mathrm{loc}}(\mathbb{R}^3, \mathbb{R}^3)$ due to possible singularities on $\Sigma$.\\
		Now, let $A_0$ be one of the components $A_\tau,A_\rho,A_\zeta$. To prove that $\nabla A_\tau, \nabla  A_\rho, \nabla A_\zeta \in L^1_{loc}(\R^3,\R^3)$, we show that $\frac{\partial A_0}{\partial x_i}|_{\R^3\setminus\Sigma}$ actually coincides with the distributional derivative of $A_0$ in the whole $\R^3$, namely that, for every $B\in\cC_0^{\infty}(\R^3,\R^3)$,
		\[
		\int_{\R^3}\left.\frac{\partial A_0}{\partial x_i}\right|_{\R^3\setminus\Sigma}B \,dx
		=-\int_{\R^3}A_0\frac{\partial B}{\partial x_i}\,dx.
		\]
		Thus, let us  set $A_0=A_1+A_2$ for some $A_1\in L^6(\R^3,\R^3)$ and $A_2\in L^{q^*}(\R^3,\R^3)$. Take any $B\in\cC_0^{\infty}(\R^3,\R^3)$ and $R>0$ such that $B=0$ for $|x|\geq R$. Observe that there is a constant $C>0$ such that, for any $\eps>0$,
		\begin{equation}\label{eq:etaestim}
			\begin{split}
				\frac1\eps\int_{r\leq\eps,|x_3|\leq R}|A_0|\,dx
				& 	\leq \frac1\eps\int_{r\leq\eps,|x_3|\leq R}|A_1|\,dx
				+\frac1\eps\int_{r\leq\eps,|x_3|\leq R}|A_2|\,dx\\
				& \leq
				C\Big(\eps^{\frac23}|A_1|_6+\eps^{\frac{5q-6}{3q}}|A_2|_{q^*}\Big).
			\end{split}
		\end{equation}
		Now, taking a smooth function $\eta_\eps\in\cC^\infty([0,\infty),[0,1])$ such that $\eta_\eps=0$ for $r\leq\eps/2$, $\eta_\eps=1$ for $r\geq\eps$ and $\eta_\eps'(r)\leq 4/\eps$ and setting $B_\eps(x):=\eta_\eps(r) B(x)$, we have that
		\begin{equation}
			\label{last}
			\int_{\R^3}A_0\frac{\partial B_\eps}{\partial x_i}\,dx
			=\int_{\R^3}\eta_\eps A_0\frac{\partial B}{\partial x_i}\,dx
			+\int_{\R^3}A_0B\frac{\partial \eta_\eps}{\partial x_i}\,dx.
		\end{equation}
		Then
		\[
		\left|  \int_{\R^3}A_0\frac{\partial B_\eps}{\partial x_i}\,dx-\int_{\R^3}A_0\frac{\partial B}{\partial x_i}\,dx \right|
		\leq C \Big(\int_{r\leq\eps,|x_3|\leq R}|\eta_\eps - 1| |A_0|\,dx +\frac1\eps\int_{r\leq\eps,|x_3|\leq R}|A_0|\,dx\Big)
		\]
		and, since
		\begin{align*}
			\int_{r\leq\eps,|x_3|\leq R}|\eta_\eps - 1| |A_0|\,dx
			&\leq
			\int_{r\leq\eps,|x_3|\leq R}|\eta_\eps - 1| |A_1|\,dx
			+\int_{r\leq\eps,|x_3|\leq R}|\eta_\eps - 1| |A_2|\,dx\\
			&\leq
			\Big(\int_{r\leq\eps,|x_3|\leq R}|\eta_\eps - 1|^\frac{6}{5} \,dx\Big)^\frac{5}{6} |A_1|_6\\
			&\qquad
			+ \Big(\int_{r\leq\eps,|x_3|\leq R}|\eta_\eps - 1|^\frac{3q}{4q-3}\,dx\Big)^\frac{4q-3}{3q} |A_2|_{q^*}\\
			&\leq
			C \big( \eps^\frac{5}{3} |A_1|_6
			+ \eps^\frac{2(4q-3)}{3q} |A_2|_{q^*}  \big),
		\end{align*}
		in view of \eqref{eq:etaestim} we infer that 
		$$\lim_{\eps\to 0^+}\int_{\R^3}A_0\frac{\partial B_\eps}{\partial x_i}\,dx=\int_{\R^3}A_0\frac{\partial B}{\partial x_i}\,dx \in\R.$$
		On the other hand, by Lebesgue's Theorem
		\[
		\lim_{\varepsilon\to 0^+}\int_{\R^3}\eta_\eps A_0\frac{\partial B}{\partial x_i}\,dx
		= \int_{\R^3} A_0\frac{\partial B}{\partial x_i}\,dx
		\]
		and an easy computation shows that
		\[
		\lim_{\varepsilon\to 0^+}\int_{\R^3}A_0B\frac{\partial \eta_\eps}{\partial x_i}\,dx=0
		\]
		(see \cite[formula (26)]{BenForAzzAprile}). Thus, by \eqref{last}, we can conclude.
		\\
		Finally, \eqref{curlort} follows from direct computations.
	\end{proof}
	
	Let us introduce now the following space
	\[
	\cD:=\{A\in\cL: A=\widehat{A}+\widetilde{A},\;\widehat{A}\in \cD^{1,2}(\R^3,\R^3),\widetilde{A}\in \cD^{1,q}(\R^3,\R^3)\},
	\]
	equipped with the norm
	$$\|A\|_\cD:=\inf\{|\nabla \widehat{A}|_2 + |\nabla \widetilde{A}|_q:A=\widehat{A}+\widetilde{A},\;\widehat{A}\in \cD^{1,2}(\R^3,\R^3),\widetilde{A}\in \cD^{1,q}(\R^3,\R^3)\},$$
	where $\D^{1,k}(\R^3,\R^3)$ the completion of $\cC_0^{\infty}(\R^3,\R^3)$ with respect to the norm $|\nabla \cdot |_k$.\\
	Clearly, there is a continuous embedding of $\cD$ into $\cL$ due to the classical Sobolev embedding of $\D^{1,k}(\R^3,\R^3)$ into $L^{k^*}(\R^3,\R^3)$ with $k^*=3k/(3-k)$.

	Let
	$\cD_G$ be the set of fixed points in $\cD$  with respect to the action \eqref{actionG}. In view of Proposition \ref{prop:decomposition}, we observe that
	any $A\in \D_G$ has a unique decomposition 
	$$A=A_\tau+A_\rho+A_\zeta$$ with summands of the form \eqref{eq:summands}. Moreover, since the elements of $\cD$ have gradient in $L_{loc}^1(\R^3,\R^3)$, using also \eqref{eq:nablasquare}, we get that $A_\tau,A_\rho,A_\zeta\in\cD_G$.
	
	Therefore we define
	\begin{equation}\label{actionS}
		S:\D_G\to \D_G, \quad S(A_\tau+A_\rho+A_\zeta) := A_\tau-A_\rho-A_\zeta,
	\end{equation}
	which,  taking into account \eqref{eq:nablasquare}, is a linear isometry and $S^2=\id$,
	and let
	\begin{equation*}\label{eq:A^2}
		\cA:=
		\big\{A\in \D_G:SA=A\big\}.
	\end{equation*}
	Observe that, if $A\in\cA$, then $A_\rho+A_\zeta=0$.
	
	On $\cA$ we can prove the following result.
	\begin{Lem} If $q\in(6/5,2)$, then for any $A\in\cA$, $\div A=0$ in the sense of distributions and
		$$\|A\|=\inf\big\{|\curlop \widehat{A}|_2 + |\curlop \widetilde{A}|_q:\; A=\widehat{A}+\widetilde{A},\;\widehat{A}\in \cD^{1,2}(\R^3,\R^3),\widetilde{A}\in \cD^{1,q}(\R^3,\R^3)\big\},$$
		defines an equivalent norm in $\cA$.
	\end{Lem}	
	\begin{proof}
		Let $A\in\cA$. Since $\div A_\tau =0$ in the sense of distributions, then it is clear that $\div A=0$.
		Moreover, since $|\nabla\times A_\tau|^2 \leq 2 |\nabla A_\tau|^2$ a.e. in $\R^3$, then $\|A\|\leq C \|A\|_\cD$.\\	 
		Now, let $\widehat{A}\in \cD^{1,2}(\R^3,\R^3)$, $\widetilde{A}\in \cD^{1,q}(\R^3,\R^3)$ be such that $A=\widehat{A}+\widetilde{A}$.\\
		In view of the Helmholz decomposition \cite[Lemma 2.2]{MedSz}, $\widehat{A}=\widehat{v}+\widehat{w}$ for some unique $\widehat{v},\widehat{w}\in \big\{u\in L^{6}(\R^3,\R^3): \curlop u\in L^2(\R^3,\R^3)\big\}$ such that $\div \widehat{v}=0$ and $\curlop \widehat{w}=0$.\\
		In a similar way we find the Helmholz decomposition \cite{MederskiSob_curl}	of $\widetilde{A}=\widetilde{v}+\widetilde{w}$ for some unique $\widetilde{v},\widetilde{w}\in \big\{u\in L^{q^*}(\R^3,\R^3): \curlop u\in L^q(\R^3,\R^3)\big\}$
		such that $\div \widetilde{v}=0$ and $\curlop \widetilde{w}=0$. Moreover $|\curlop \widetilde{v}|_q\geq C|\nabla \widetilde{v}|_q$.
		\\
		Observe that
		\begin{align*}
			|\curlop \widehat{A}|_2^2
			&=
			|\curlop \widehat{v}|_2^2
			=|\nabla \widehat{v}|_2^2,\\
			|\curlop \widetilde{A}|_q^q
			&=
			|\curlop \widetilde{v}|^q
			\geq C |\nabla \widetilde{v}|_q^q, 
		\end{align*}
		and so
		\begin{equation}\label{eq:ineqAv}
			|\curlop \widehat{A}|_2 + |\curlop \widetilde{A}|_q\geq |\nabla \widehat{v}|_2 + C|\nabla \widetilde{v}|_q\geq \min\{1,C\}\|\widehat{v}+\widetilde{v}\|_\cD.
		\end{equation}
		Again by \cite{MedSz,MederskiSob_curl}, $\widehat{w}=\nabla \widehat{\vp}$ and $\widetilde{w}=\nabla \widetilde{\vp}$ 	for some $\widehat{\vp}\in\cD^{1,6}(\R^3)$ and $\widetilde{\vp}\in\cD^{1,q^*}(\R^3)$. Since $\div A=0$, we get
		$$0=\div(\widehat{w}+\widetilde{w})=\Delta(\widehat{\vp}+\widetilde{\vp}).$$
		Therefore, $\widehat{\vp}+\widetilde{\vp}$ is a harmonic function, so are all components of $\widehat{w}+\widetilde{w}=\nabla(\widehat{\vp}+\widetilde{\vp})$. 
		By the mean-value formula, we infer that 
		$$|\widehat{w}(y)+\widetilde{w}(y)|\leq \frac{C}{R^3}\int_{B(y,R)}|\widehat{w}|+|\widetilde{w}|\,dx\leq \frac{C}{R^3}\big(R^{\frac{5}{2}}|\widehat{w}|_6+R^{\frac{4q-3}{q}}|\widetilde{w}|_{q^*}\big)$$
		for any $y\in\R^3$ and $R>0$. Letting $R\to+\infty$ we get that $\widehat{w}+\widetilde{w}=0$. Hence,
		$A=\widehat{v}+\widetilde{v}$ and by \eqref{eq:ineqAv} we obtain that $\|A\|\geq \min\{1,C\}\|A\|_\cD$. Thus, we obtain the desired equivalence of the norms in $\cA$.
	\end{proof}

	Now we give the following properties on the functional $\cJ$.
	
	\begin{Lem}\label{LemJC^1}
		Let $J\in\cA^*$. Then $\cJ:\cA\to\R$ is a strictly convex, continuous, coercive, and Fr\'echet  differentiable functional. Moreover, for any $A,B\in\cA$,
		\begin{equation*}%\label{eq:DevFormula}
			\cJ'(A)[B]=\int_{\R^3} \left(\frac{\curlop A}{\big(1+|\curlop A|^2\big)^{1-q/2}} \cdot \curlop B\right) \,dx
			- \langle J,B\rangle,
		\end{equation*}
		where  $\langle\ , \ \rangle$ in \eqref{eq:functionalBI} denotes the duality pairing between $\cA^*$ and $\cA$, and $\J'(A)\in \cA^*$.
	\end{Lem}
	\begin{proof}
		Firstly we show that $\cJ$ is of class $\cC^1$ on $\cA$.
		Let $f:[0,+\infty)\to \R$, $f(t)=(1+t^2)^{q/2}-1$. Observe that there is a constant $c\in (0,1)$ such that, for any $t\geq 0$,
		\begin{equation}\label{eq:t2qest}
			c\min\{t^2,t^q\}\leq f(t) \leq \min\{t^2,t^q\}
		\end{equation}
		and
		\[
		\frac1q f'(t)=\frac{t}{(1+t^2)^{1-q/2}} \leq \min\{t,t^{q-1}\}.
		\]
		In view of \cite[Proposition 3.8]{BadPisRol}, the operator
		$$\cF:L^2(\R^3,\R^3)+L^q(\R^3,\R^3)\to L^1(\R^3)$$
		given by
		$$\cF(u)(x):=\frac1q f(|u(x)|),\quad\hbox{for }u\in L^2(\R^3,\R^3)+L^q(\R^3,\R^3),\; x\in\R^3$$
		is well-defined and of class $\cC^1$, with Fr\'echet derivative $\cF'(u)$ given by
		$$\cF'(u)[h]=\frac{u}{(1+|u|^2)^{1-q/2}}\cdot h,\quad\hbox{for }u,h\in L^2(\R^3,\R^3)+L^q(\R^3,\R^3).$$
		Observe that, if $A\in\cA$, then  $\curlop A\in L^2(\R^3,\R^3)+L^q(\R^3,\R^3)$, and then  $\cJ$ is well-defined and of class $\cC^1$ on $\cA$. Now
		it remains to prove the strict convexity and the coercivity of $\cJ$.\\
		Since $f''(t)>0$ for $t\geq 0$, $\cJ$ is strictly convex.\\
		Now, we derive the necessary lower bounds to ensure coercivity.
		Let
		$\{A_n\}\subset\cA$ and $\|A_n\|\to+\infty$ as $n\to+\infty$, $\Om_n:=\{x\in\R^3: |(\curlop A_n) (x)|>1\}$, and note that $\Omega_n$ is of finite measure. Denoting by $\chi_\Omega$ the characteristic function of $\Omega\subset\R^3$, by \eqref{eq:t2qest}, we get
		\begin{align*}
			\cJ(A_n)
			&\geq
			\frac{c}{q}\int_{\R^3} \min\{|\curlop A_n|^2,|\curlop A_n|^q\}\, dx-\|J\|_{\cA^*}\|A_n\|\\
			&=
			\frac{c}{q}\int_{\R^3\setminus\Omega_n} |\curlop A_n|^2\, dx+\frac{c}{q}\int_{\Omega_n} |\curlop A_n|^q\, dx-\|J\|_{\cA^*}\|A_n\|\\
			&\geq
			\frac{c}{2q}\min\big\{
			|(\curlop A_n)\chi_{\R^3\setminus\Omega_n}|_{2}^2
			+|(\curlop A_n)\chi_{\Omega_n}|^2_{q},
			|(\curlop A_n)\chi_{\R^3\setminus\Omega_n}|_{2}^q
			+|(\curlop A_n)\chi_{\Omega_n}|^q_{q}\big\}\\
			&\qquad
			-\|J\|_{\cA^*}\|A_n\|\\
			&\geq
			\frac{c}{2q}\min\big\{\|A_n\|^2,\|A_n\|^q\big\}-\|J\|_{\cA^*}\|A_n\|,
		\end{align*}
		where the last inequality follows from \cite[Proposition 2.13]{BadPisRol}.\\
		Thus, since $q>1$ we infer that $\cJ(A_n)\to+\infty$ as $n\to+\infty$.
	\end{proof}
	
	Now we are ready to prove our  main result.
	
	\begin{proof}[Proof of Theorem~\ref{th:mainBI}]
		First observe that, since $\cA$ is a closed subset of the reflexive space $\cL$, then it is reflexive too. Then for 
		$q\in (6/5,2)$, we are able to use the direct methods of the Calculus of Variations to show that there exists the global minimizer.\\
		Thus, since $\cJ$ is coercive and continuous,
		$$c=\inf_\cA \cJ>-\infty.$$
		Let us take a minimizing sequence $\{A_n\}\subset\cA$, namely such that $\cJ(A_n)\to c$. Then $\{A_n\}$ is bounded and $A_n\weakto A_0$ for some $A_0\in\cA$. 
		Since $\cD^{1,2}(\R^3,\R^3)$ and $\cD^{1,q}(\R^3,\R^3)$ are compactly embedded into $L^2_{loc}(\R^3,\R^3)$, passing to a subsequence we may assume that  $A_n\to A_0$ a.e. in $\R^3$. Thus, by the Fatou's lemma $J(A_0)=c$ and so $A_0$ is a critical point of $\cJ$ in $\cA$.\\
		Moreover $A_0$ is the unique global minimizer of $\J$ in $\cA$ since, the strict convexity implies that, for all $B\in\cA\setminus\{0\}$,
		\[
		\J(A_0)=\J(A_0)+\J'(A_0)[B]<\J(A_0+B)
		\]
		and, if $J\neq0$, then $A_0\neq 0$.\\
		Now, since $\cJ$ is invariant with respect to the $G$-action given by \eqref{actionG} as well as with respect to $S$-action given by \eqref{actionS}, by the Palais principle of symmetric criticality \cite{Palais}, any critical point of $\cJ$ on $\cA$ is a critical point of the unconstrained functional $\cJ$ in $\cD$. Since $\cD$ contains $\cC_0^{\infty}(\R^3,\R^3)$, we conclude.
	\end{proof}

	{\bf Acknowledgements.}
	We thank the Reviewers for their  comments, which improved the clarity and presentation of this work.
	P. d'Avenia is member of INdAM-GNAMPA and he was partly supported by PRIN 2017JPCAPN {\em Qualitative and quantitative aspects of nonlinear PDEs}, by European Union - Next Generation EU - PRIN 2022 PNRR
	P2022YFAJH {\em Linear and Nonlinear PDE's: New directions and
		Applications}, by INdAM-GNAMPA 2024 project {Metodi variazionali
		per alcune equazioni di tipo Choquard}, and by the Italian Ministry of University and Research under the Programme {\em Department of Excellence} L. 232/2016 (Grant No. CUP - D93C23000100001).
	J. Mederski was partly supported by the National Science Centre, Poland (Grant No. 2017/26/E/ST1/00817).

\end{document}